\documentclass{amsart}

\usepackage{amsfonts,amssymb,amscd,amsmath,enumerate,verbatim,calc,cite}

\newcommand{\Ann}{\operatorname{Ann}}
\newcommand{\depth}{\operatorname{depth}}
\newcommand{\grade}{\operatorname{grade}}
\newcommand{\height}{\operatorname{height}}

\newcommand{\tra}{\operatorname{Tr}}

\newcommand{\N}{{\mathbb N}}

\theoremstyle{plain}
\newtheorem{theorem}{Theorem}

\newtheorem{lemma}[theorem]{Lemma}

\newtheorem{proposition}[theorem]{Proposition}

\theoremstyle{definition}
\newtheorem*{ack}{Acknowledgements}

\theoremstyle{remark}
\newtheorem{remark}[theorem]{Remark}
\newtheorem{Example}[theorem]{Example}

\newcounter{hours}\newcounter{minutes}
\newcommand{\printtime}{%
        \setcounter{hours}{\time/60}%
        \setcounter{minutes}{\time-\value{hours}*60}%
        \thehours\,h\ \theminutes\,min}

\begin{document}

\title[ CM Ideal]{ On Cohen-Macaulayness and depth of ideals in invariant rings}
\date{\today \printtime}

\author{ Martin Kohls}
\address{Technische Universit\"at M\"unchen \\
 Zentrum Mathematik-M11\\
Boltzmannstrasse 3\\
 85748 Garching, Germany}
\email{kohls@ma.tum.de}

\author{M\"uf\.it Sezer}
\address { Department of Mathematics, Bilkent University,
 Ankara 06800 Turkey}
\email{sezer@fen.bilkent.edu.tr}
\thanks{Second author is supported by a grant from T\"ubitak:112T113}


\begin{abstract}
We investigate the presence of  Cohen-Macaulay ideals in invariant
rings and show  that an ideal of an invariant ring corresponding
to a modular representation of a $p$-group is not Cohen-Macaulay
unless the invariant ring itself is. As an intermediate result, we obtain that
non-Cohen-Macaulay factorial rings cannot contain Cohen-Macaulay ideals.
For modular cyclic groups of prime order,
we show that the quotient of the invariant ring modulo the transfer ideal is
always Cohen-Macaulay, extending a result of Fleischmann.
\end{abstract}

\maketitle

\section{Introduction}
The depth and Cohen-Macaulay property of invariant rings have
always been among  the major interests of invariant theorists, see
the references below. In this paper, we consider ideals of
invariant rings (as  modules over the latter), and investigate
their depth and Cohen-Macaulayness. The original goal of this
paper was to find  filtrations of the invariant rings with
Cohen-Macaulay quotients  (a weakening of being ``sequentially
Cohen-Macaulay'' as introduced in \cite[section
III.2]{StanleyCombinatorics}). However, the results of this paper
show that in many cases, invariant rings fail to contain any
Cohen-Macaulay ideal, so the goal is missed in the first step
already. Before we go into more details, we fix our setup.
 Let $V$ be a finite dimensional representation of a group $G$ over
a field $K$.  The representation is called modular if the
characteristic of $K$ divides the order of $G$. Otherwise, it is
called non-modular.  There is an induced action on the symmetric
algebra $K[V]:=S(V^{*})$ given by $\sigma (f)=f\circ \sigma^{-1}$
for $\sigma \in G$ and $f\in K[V]$. We let
\[
K[V]^G:=\{f\in  K[V] \mid \sigma(f)=f   \; \text{ for all }
\sigma\in G\}
\]
 denote the
 subalgebra of invariant polynomials in $K[V]$. For any
 non-modular representation, $K[V]^G$ is a Cohen-Macaulay ring \cite{HochsterEagon}. In the modular case,  on
 the other hand, $K[V]^G$ almost always fails
 to be Cohen-Macaulay, see \cite{KemperOnCM}. The depth of
 $K[V]^G$  has attracted much attention and has been determined for various families
 of  representations, see for example
 \cite{CampEtAl,FleischmannCohomConnectivity,SmithKoszul,ElmerAssoc,KemperSurvey}. In this
 paper we consider ideals of $K[V]^{G}$ as modules over $K[V]^G$. We
 show that, if $V$ is a modular representation of a $p$-group,
 then $K[V]^G$ does not contain a  Cohen-Macaulay ideal unless
 $K[V]^G$ is Cohen-Macaulay itself. In fact, our results hold in a
 broader generality. We first show that a Cohen-Macaulay ideal in
 an  affine domain can not have height bigger than one. If, in
 addition, the  affine domain is factorial, then only principal
ideals can be
 Cohen-Macaulay. So we get the desired implication for the groups and their representations whose invariants are factorial.
We also include an example that shows that the condition that the
affine domain is factorial can not be dropped. We then restrict to
modular representations of a cyclic group of prime order. Our main
result here is that the quotient $K[V]^{G}/I^{G}$ of the invariant
ring modulo the transfer ideal $I^{G}$ is Cohen-Macaulay. Note
that this extends results of Fleischmann \cite{MR1714612} in this
case, namely that $K[V]^{G}/\sqrt{I^{G}}$ is Cohen-Macaulay, and
that $\sqrt{I^{G}}=I^{G}$ if $V$ is projective. This also allows
us to compute the depth of the transfer ideal. We end with a
reduction result that reduces computing the depth of a
$K[V]^{G}$-module to computing a grade of the transfer ideal.

We refer the reader to \cite{MR1249931,EddyDavidBook,DerksenKemper} for
more background in modular invariant theory.

\section{Preliminaries}
In this section we summarize our notation as well as some basic
results that we use in our computations.
Let $R=\bigoplus_{d=0}^{\infty}R_{d}$ be a graded affine
$K$-algebra such that $R_{0}=K$, and
$M=\bigoplus_{d=0}^{\infty}M_{d}$ a finitely generated graded
nonzero $R$-module. We call $R_{+}:=\bigoplus_{d=1}^{\infty}R_{d}$ the
maximal homogeneous ideal of $R$. A sequence of homogeneous
elements $a_{1},\ldots,a_{k}\in R_{+}$ is called $M$-regular if
each $a_{i}$ is a nonzero divisor on $M/(a_{1},\ldots,a_{i-1})M$
for $i=1,\ldots,k$. For a homogeneous ideal $I\subseteq R_{+}$,
the maximal length of an $M$-regular sequence lying in $I$
 is called the grade of $I$ on $M$, denoted by $\grade (I,M)$.
Furthermore, one calls
 $\depth(M):=\grade(R_{+},M)$ the depth of $M$.
Recall that we have
 $\depth(M)\le \dim (M)$ (where $\dim(M):=\dim(R/\Ann_{R}(M))$ denotes the
 Krull dimension), and $M$ is called \emph{Cohen-Macaulay} if equality
 holds.

 By Noether-Normalization, $R$ contains a homogeneous system of
 parameters (hsop), i.e., algebraically independent homogeneous elements
 $a_{1},\ldots,a_{n}\in R$ such that $R$ is finitely generated as a module
 over the (polynomial) subalgebra $A:=K[a_{1},\ldots,a_{n}]$. Note that
 $n=\dim(R)$ is uniquely determined. Any subset of an hsop is called a partial
 hsop (phsop). If $R$ is also a domain, then homogeneous elements
 $a_{1},\ldots,a_{k}\in R_{+}$ form a phsop if and only if
 $\height(a_{1},\ldots,a_{k})=k$, (see \cite[Lemma 1.5]{KemperOnCM}, \cite[Theorem A. 16]{MR1251956}). Note that $M$ is also an
 $A$-module, and from
 the graded Auslander-Buchsbaum formula \cite[Exercise 19.8]{Eisenbud}
 we get that $M$ is free as an $A$-module if and only if its depth
 as an $A$-module is equal to $\dim (A)$. But since the depths of $M$
 as an $A$- and as an $R$-module are equal (see \cite[Lemma
 3.7.2]{DerksenKemper} or \cite[Exercise
 1.2.26]{MR1251956}), this is also equivalent to the condition that
 the depth of $M$ as an $R$-module is $\dim (R)=\dim (A)$. In other words: $M$
 is free as an $A$-module if and only if $M$ is Cohen-Macaulay and
 $\dim(M)=\dim(R)$, i.e., $M$ is \emph{maximal Cohen-Macaulay}.

We include the following standard facts about depth for the reader's
convenience.

\begin{lemma}\cite[Proposition 1.2.9]{MR1251956}\label{StandardIneqs}
Assume that  $I$ is a homogeneous nonzero proper ideal of the
graded affine ring $R$. Then we have the following inequalities.
\begin{itemize}
\item[(a)] $\depth(R)\ge \min\{\depth(I),\depth(R/I)\}$.
\item[(b)] $\depth(I)\ge\min\{\depth(R),\depth(R/I)+1\}$.
\item[(c)] $\depth(R/I)\ge\min\{\depth(I)-1,\depth(R)\}$.
\end{itemize}
\end{lemma}

This lemma implies that $\depth(I)$ and $\depth(R/I)$ are often
strongly related:

\begin{lemma}\label{StandardCor1}
Assume that $I$ is a homogeneous nonzero proper ideal of the
graded affine ring $R$.
\begin{itemize}
\item[(a)] If one of the following conditions
\begin{itemize}
\item[(i)] $\depth(R)>\depth(I)$, \item[(ii)]
$\depth(R)>\depth(R/I)$, \item[(iii)] $R$ is a Cohen-Macaulay
domain
\end{itemize}
holds, then
\[
\depth(I)=\depth(R/I)+1.
\]
\item[(b)] If $\depth(I)>\depth(R)$, then $\depth(R/I)=\depth(R)$.
\item[(c)] If $\depth(R/I)>\depth(R)$, then $\depth(I)=\depth(R)$.
\end{itemize}
\end{lemma}

\begin{proof}
(a) Assume first $\depth(R)>\depth(I)$. From Lemma
\ref{StandardIneqs} (b) it follows that $\depth(I)\ge
\depth(R/I)+1$, and from Lemma \ref{StandardIneqs} (c)  we get
$\depth(R/I)\ge \depth(I)-1$, implying the desired equality.
Secondly, assume $\depth(R)>\depth(R/I)$. From Lemma
\ref{StandardIneqs} (b) it follows that $\depth(I)\ge
\depth(R/I)+1$, and from Lemma \ref{StandardIneqs} (c) we have
$\depth(R/I)\ge \depth(I)-1$ so we obtain the result again.
Finally assume $R$ is a Cohen-Macaulay domain. As $R$ is a domain
and $I\ne \{0\}$, it follows that $\dim(R/I)<\dim(R)$. Hence we
have the inequality $\depth(R/I)\le\dim(R/I)<\dim(R)=\depth(R)$,
so the assertion follows from (ii).

Statement (b) follows similarly from Lemma \ref{StandardIneqs} (a) and
(c). Statement (c) follows from Lemma \ref{StandardIneqs} (a) and
(b).
\end{proof}

For example, if $R$ has positive depth, then the homogeneous
maximal ideal always has depth one by the above lemma, as its
quotient is zero-dimensional. Now we note that for any given
number $1\le k\le \depth(R)$, there exists an ideal of depth  $k$:

\begin{lemma}\label{DepthRegSeq}
Assume that the homogeneous elements $a_{1},\ldots,a_{k}$ of positive degree
form a regular sequence of  $R$. Then
\[
\depth((a_{1},\ldots,a_{k})R)=\depth(R)+1-k.
\]
\end{lemma}

\begin{proof}
We have $\depth(R/(a_{1},\ldots,a_{k})R)=\depth(R)-k<\depth(R)$, and the result
follows from the previous lemma.
\end{proof}

\section{Cohen-Macaulay ideals in affine domains}
The main result of this section is Theorem  \ref{factorialCM}
where it is shown that only principal ideals can be Cohen-Macaulay
in factorial affine domains. Nevertheless even when the affine
domain is not factorial, the height of a Cohen-Macaulay ideal can
be at most one.

\begin{lemma}\label{height2notCM}
Assume that $R$ is a graded affine domain, and $I\ne R$ a
homogeneous ideal of height at least $2$. Then $I$ is not
Cohen-Macaulay as an $R$-module.
\end{lemma}

\begin{proof}
As $I$ is homogeneous of height at least two, it contains a phsop
$p,q$ of $R$. We extend this phsop to an hsop of $R$ and consider
the $K$-subalgebra $A$ of $R$ generated by this hsop, i.e.,
$A=K[hsop]=K[p,q,\ldots]$. Then $A$ is isomorphic to a polynomial
ring over $K$ in $\dim(R)$ variables.  Assume by way of
contradiction that $I$ is Cohen-Macaulay as an $R$-module. Then
$I$ is free as an $A$-module, i.e., there exist elements
$g_{1},\ldots,g_{m}\in I$ such that $I=\bigoplus_{i=1}^{m}Ag_{i}$.
As $p,q$ are elements of $I$, we can find unique elements
$a_{i},b_{i}\in A$ for $i=1,\ldots,m$ such that
$p=\sum_{i=1}^{m}a_{i}g_{i}$ and $q=\sum_{i=1}^{m}b_{i}g_{i}$.
Multiplying both equations with $q$ and $p$ respectively, we get
$\sum_{i=1}^{m}(qa_{i})g_{i}=pq=\sum_{i=1}^{m}(pb_{i})g_{i}$. As
$qa_{i},pb_{i}\in A$ and $g_{1},\ldots,g_{m}$ is a free $A$-basis
of $I$, we get $qa_{i}=pb_{i}$ for all $i=1,\ldots,m$. As $p,q$
are different variables in the polynomial ring $A$, it follows
$p|a_{i}$ and $q|b_{i}$ for all $i$. Therefore there exist
$b_{i}'\in A$ such that $b_{i}=qb_{i}'$ for $i=1,\ldots,m$. Hence
we get $q=\sum_{i=1}^{m}b_{i}g_{i}=\sum_{i=1}^{m}qb_{i}'g_{i}$,
and as we are in a domain dividing by $q$ yields
$1=\sum_{i=1}^{m}b_{i}'g_{i}\in \bigoplus_{i=1}^{m}Ag_{i}=I$. This
implies $I=R$, contradicting the hypothesis $I\ne R$ of the lemma.
\end{proof}

\begin{theorem}\label{factorialCM}
Assume that $R$ is a graded factorial affine domain.
 If $I \ne R$ is a homogeneous
ideal which is not principal, then $I$ is not Cohen-Macaulay as an
$R$-module.
 Therefore,  if $R$  is not Cohen-Macaulay, then $R$ does not contain any nonzero
homogeneous Cohen-Macaulay ideal (as an $R$-module).

\end{theorem}

\begin{proof}
 Assume by way of contradiction that $I$ is Cohen-Macaulay and not principal.  Let $a_{1},\ldots,a_{n}$
 denote a finite set of generators of $I$. As we are in a
factorial ring, we can consider the greatest common divisor $d$ of
those elements. Then we have $I=(a_{1},\ldots,a_{n})\subsetneq
(d)$, where the inclusion is strict as $I$ is not a principal
ideal by assumption. As $R$ is a domain and $d\mid a_{i}$ for all
$i$, the elements $\frac{a_{i}}{d}\in R$ are well defined, and we
can consider the ideal
$J:=(\frac{a_{1}}{d},\ldots,\frac{a_{n}}{d})=\frac{1}{d}I$. Note
that from $I\subsetneq (d)$ it follows that $J\subsetneq (1)=R$,
so $J$ is a proper ideal of $R$. Multiplication by $d$ yields an
$R$-module isomorphism from $J$ to $I$, and therefore $J$ is also
Cohen-Macaulay as an $R$-module. From Lemma \ref{height2notCM} it
follows that the height of $J$ is at most $1$. But $R$ is a domain
and $J\ne 0$, so the height of $J$ is $1$. It follows that there
exists a prime ideal $\wp$ of $R$ of height one such that
$I\subseteq \wp$. As $R$ is factorial, height one primes are
principal, and so $\wp$ is generated by a prime element $p$, so we
have $J\subseteq \wp=(p)$, which implies that $p$ is a common
divisor of $\frac{a_{1}}{d},\ldots,\frac{a_{n}}{d}$. This is a
contradiction, as $d$ is the greatest common divisor of
$a_{1},\ldots,a_{n}$.

Now the second assertion of the theorem  follows from the first and the
fact that principal ideals of $R$ are isomorphic to $R$ as
$R$-modules.
\end{proof}

We demonstrate two examples of affine domains with non-principal
Cohen-Macaulay ideals. First one is a Cohen-Macaulay ring, the
second one is not. Therefore,  a non-Cohen-Macaulay ring may
contain a Cohen-Macaulay ideal, and  the hypothesis of $R$ being
factorial can not be dropped out in the previous theorem.

\begin{Example}
Consider the subalgebra $R=K[x^{2},y^{2},xy]$ of the polynomial
ring $K[x,y]$ in two variables. Note that $R$ is not factorial as
the equality $x^{2}\cdot y^{2}=(xy)\cdot(xy)$ shows. We claim that
the ideal $I=(x^{2},xy)$ of $R$ is Cohen-Macaulay and not
principal. Clearly $I$ is not principal,  because $R$ is a graded
ring that starts in degree $2$.  We now consider the hsop
$x^{2},y^{2}$ of $R$ and the subalgebra $A=K[x^{2},y^{2}]$
generated by the hsop. We claim that we have the direct sum
decompositions $R=A\oplus Axy$ and $I=Ax^{2}\oplus A xy$. In both
cases, the directness of the sum follows as in the first summands
all $x$ degrees are even, while in the second summands all $x$
degrees are odd. As both sums contain the respective ideal
generators, it only remains to show that both sums are invariant
under multiplication with $xy$. For the sum for $R$, this follows
from $xy\cdot xy=x^2y^2\in A$. For the sum for $I$, we have
$xy\cdot x^{2}=x^{2}\cdot xy\in Axy$ and $xy\cdot xy=y^{2}\cdot
x^{2}\in Ax^{2}$. Therefore $R$ and $I$ are free $A$-modules, so
$R$ and $I$ are Cohen-Macaulay as $R$-modules. Also note that
$I\subseteq \sqrt{(x^{2})}$, as $(xy)^{2}=y^{2}\cdot x^{2}\in
(x^{2})$. Thus, $\height(I)\le\height((x^{2}))=1$, as  predicted
by Lemma \ref{height2notCM}.
\end{Example}

We state the following example  as a proposition.

\begin{proposition}
Consider the subalgebra $R:=K[x^{4},x^{3}y,xy^{3},y^{4}]$ of the
polynomial ring $K[x,y]$. Then the ideal $I:=(x^{4},x^{3}y)$ of
$R$ is of height one and Cohen-Macaulay as an $R$-module. (While
$R$ is not Cohen-Macaulay and not factorial).
\end{proposition}

\begin{proof}
First note that $R$ is well known to be non-Cohen-Macaulay, see
\cite[Example 2.5.4]{DerksenKemper}, and as $x^{4}\cdot
y^{4}=(x^{3}y)\cdot(xy^{3})$, $R$ is also not factorial. Also
since $(x^{3}y)^{4}=y^{4}x^{8}\cdot x^{4}\in(x^{4})$ we have
$I\subseteq \sqrt{(x^{4})}$, which shows that the height of $I$ is
$1$. We consider the subalgebra $A:=K[x^{4},y^{4}]$ of $R$
generated by an hsop, and we will show that $I$ is a free
$A$-module, which implies that $I$ is Cohen-Macaulay as an
$R$-module. We set
\[
a:=x^{4},\quad p:=x^{3}y,\quad q:=xy^{3},\quad b:=y^{4}
\]
and claim that
\[
I=(a,p)=Aa\oplus Aaq\oplus Ap\oplus Ap^{2}.
\]
The inclusion ``$\supseteq$'' is clear. We first show that the sum
on the right hand side is indeed direct. Let $\epsilon$ denote the
map from the set of monomials of $K[x,y]$ to $\N_{0}^{2}$ given by
$\epsilon(x^{i}y^{j}):=(i,j)$. We compute the epsilon values of
the $A$-module generators modulo $4$:
\[
\epsilon(a)=(4,0),\quad \epsilon(aq)=(5,3)\equiv(1,3), \quad
\epsilon(p)=(3,1),\quad \epsilon(p^{2})=(6,2)\equiv(2,2).
\]
As $\epsilon(m)\equiv (0,0)$ for any monomial $m$ in $A$, it
follows that the $\epsilon$-values of monomials in $Aa$, $Aaq$,
$Ap$, $Ap^{2}$ fall into different congruence classes modulo $4$,
so the sum is indeed direct.

We now verify the inclusion ``$\subseteq$''. Clearly, $a,p\in
S:=Aa\oplus Aaq\oplus Ap\oplus Ap^{2}$, and $S$ is closed under
multiplication with $a$ and $b$. It remains to show that $S$ is
closed under multiplication with $p$ and $q$, which follows from
\[
\begin{array}{ll}
p(Aa)=Aap\subseteq Ap,\quad & q(Aa)=Aaq,\\
p(Aaq)=Aax^{4}y^{4}\subseteq Aa,& q(Aaq)=Ax^{6}y^{6}=Abp^{2}\subseteq Ap^{2}, \\

p(Ap)=Ap^{2},\quad& q(Ap)=Apq=Ax^{4}y^{4}\subseteq Aa,\\
p(Ap^{2})=Ax^{9}y^{3}=Aa^{2}q\subseteq Aaq, \quad\quad&
q(Ap^{2})=Ax^{7}y^{5}=Ax^4y^4p\subseteq Ap.
\end{array}
\]
\end{proof}
\begin{remark}
We thank Roger Wiegand from whom we learned that there are
theorems that say that, for some special classes of rings, non-free
maximal Cohen-Macaulay modules have high ranks. Since the rank of
an ideal in a domain is one, and non-principal ideals are non-free, Lemma \ref{height2notCM} and Theorem
\ref{factorialCM} readily follows for such rings whose non-free
maximal Cohen-Macaulay modules are known to have a high rank.  But
we can not expect that a non-free maximal Cohen-Macaulay module
will always have rank $>1$ as the previous two examples
demonstrate.
\end{remark}
\section{Depth of ideals and quotient of the transfer in invariant rings}
We start with an application of Theorem \ref{factorialCM} to
modular invariant rings.

\begin{theorem}\label{invariantrings}
Assume that $K$ is of positive characteristic $p$ and $G$ is a
finite $p$-group. For any finite dimensional linear representation
$V$ of $G$ over $K$ such that the invariant ring $K[V]^{G}$ is not
Cohen-Macaulay, no nonzero homogeneous ideal of $K[V]^{G}$ is
Cohen-Macaulay (as a $K[V]^{G}$-module).
\end{theorem}

\begin{proof}
As a $p$-group, $G$ does not admit any non-trivial group
homomorphism $\varphi: G\rightarrow K^{\times}$, as for every
$\sigma\in G$ we have $(\varphi(\sigma))^{|G|}-1=0$, which implies
$\varphi(\sigma)=1$. Therefore $K[V]^{G}$ is factorial by a
theorem of Nakajima (see \cite[Theorem 2.11]{Nakajima}, \cite[Corollary
3.9.3]{MR1249931}). The claim now follows from Theorem
\ref{factorialCM}.
\end{proof}

For the rest of the paper we specialize to a cyclic group $G$ of
prime order $p$ equal to the characteristic of the field $K$,
which we assume to be algebraically closed. Fix a generator
$\sigma$ of $G$. There are exactly $p$ indecomposable $G$-modules
$V_1, \dots , V_{p}$ over $K$ and each indecomposable module $V_i$
is afforded by a Jordan block of dimension $i$ with 1's on the
diagonal. Let $V$ be an arbitrary $G$-module over $K$. Assume that
$V$ has $l$ summands and so we can write  $V=\sum_{1\le j\le
l}V_{n_j}$. Notice that $l=\dim V^{G}$. We also assume that none
of these summands is trivial, i.e., $n_j>1$ for $1\le j\le l$.
 We set $ K [V]=  K[x_{i,j} \mid 1\le i\le n_j, \;
1\le j\le l]$ and the action of $\sigma$ is given by $\sigma
(x_{i,j})=x_{i,j}+x_{i-1, j}$ for $1< i\le n_j$ and $\sigma
(x_{1,j})=x_{1,j}$. For $f\in K[V]$ let $N(f)=\prod_{\sigma\in
G}\sigma(f)$ denote the norm of $f$. Notice that for $1\le i\le
n_j$,
 $N(x_{i,j})$ is monic of degree $p$ as a polynomial in $x_{i,j}$. By a famous theorem of
Ellingsrud and Skjelbred \cite{Ellingsrud}, $\depth (K[V]^G)=\min \{
\dim_{K}(V^{G})+2, \dim_{K}(V)  \}=\min \{ l+2, \dim_{K}(V) \}$.
In \cite{CampEtAl}, this result is extended to some other classes
of groups, and the proof is also made more elementary and
explicit. Restricting the results of \cite{CampEtAl} to our case,
we get the following description of a maximal $K[V]^{G}$-regular
sequence, which allows to explicitly construct an ideal of a given
depth at most that of the invariant ring.

\begin{proposition}\label{regSeqCp}
A maximal $K[V]^{G}$-regular sequence   is given by
\[
\begin{array}{ccl}
x_{1,1}, x_{1,2}, N(x_{n_1,1}), \dots,
N(x_{n_l,l}) &\text{ if }&l>1;\\
x_{1,1},N(x_{2,1}),N(x_{n_{1},1}) &\text{ if
}&l=1,\,\,n_{1}>2;\\
x_{1,1},N(x_{2,1}) &\text{ if }& l=1,\,\,n_{1}=2.
\end{array}
\]
Let $I_{k}$ denote the ideal of $K[V]^{G}$ generated by the first
$k$ elements of the sequence. Then we have $\depth
I_k=\depth(K[V])^{G}+1-k$ for $1\le k\le \depth(K[V]^{G})$.
\end{proposition}

\begin{proof}
Let $b$ denote the second element of the sequence, i.e., $x_{1,2}$
or $N(x_{2,1})=x_{2,1}^{p}-x_{2,1}x_{1,1}^{p-1}$. As   $x_{1,1}$
and  $b$ are coprime in $K[V]$, and both are invariant, they form
a regular sequence in $K[V]^{G}$. Proceeding by induction, we
assume that the elements
$x_{1,1},b,N(x_{n_{1},1}),\ldots,N(x_{n_{k-1},k-1})$ form a
regular sequence for some $k<l$. Consider the standard basis
vector $e_{n_{k},k}\in V$ corresponding to the variable
$x_{n_{k},k}$. Then $e_{n_{k},k}$ is a fixed point, and
$U:=Ke_{n_{k},k}$ is a $1$-dimensional submodule of $V$. Since no
element of the regular sequence
$x_{1,1},b,N(x_{n_{1},1}),\ldots,N(x_{n_{k-1},k-1})$ contains the
variable $x_{n_{k},k}$, \cite[Corollary 17]{CampEtAl} applies to
$U$ and $x_{{n_{k},k}}$, so the regular sequence can be extended
by the element $N(x_{n_{k},k})$. Since the length of the given
sequence equals $\depth(K[V]^{G})$ in each case, we are done. The
final statement now follows from Lemma \ref{DepthRegSeq}.
\end{proof}

The transfer ideal $I^G$ is defined as the image of the transfer
map $\tra :K[V]\rightarrow K[V]^G$ given by $\tra
(f)=\sum_{i=0}^{p-1}\sigma^{i}(f)$. The transfer ideal often plays
an important role in computing the invariant ring and its
 various
aspects  have been subject to  research.
 The vanishing set of
$I^{G}$ equals the fixed point space $V^{G}$ (see \cite[Theorem
9.0.10]{EddyDavidBook}), in particular we have
$\dim(K[V]^{G}/I^{G})=\dim(V^{G})=l$.  We will show that
$K[V]^{G}/I^{G}$ is Cohen-Macaulay, which also allows us to
compute the depth of the transfer ideal. To do this we prove that
$N(x_{n_{1},1}), \ldots  ,N(x_{n_l,l})$ is a $K[V]^G/I^G$-regular
sequence. Let $f\in K[V]$ and $1\le j_1< j_2< \dots <j_t\le l$ be
arbitrary. Since $N(x_{n_{j_1},j_1})$ is a monic polynomial of
degree $p$ in $x_{n_{j_1},j_1}$, we can write
$f=q_1N(x_{n_{j_1},j_1})+r_1$, where
$\deg_{x_{n_{j_1},j_1}}r_1<p$. Next we divide $r_1$ by
$N(x_{n_{j_2},j_2})$ and we get a decomposition
$f=q_1N(x_{n_{j_1},j_1})+q_2N(x_{n_{j_2},j_2})+r_2$, where
$\deg_{x_{n_{j_1},j_1}}r_2, \deg_{x_{n_{j_2},j_2}}r_2<p$ and
$\deg_{x_{n_{j_1},j_1}}q_2<p$. In this way we get a decomposition
$$f=q_1N(x_{n_{j_1},j_1})+ \cdots +q_tN(x_{n_{j_t},j_t})+r,$$
where $\deg_{x_{n_{j_{i}},j_{i}}}r<p$ for $1\le i\le t$ and
$\deg_{x_{n_{j_{i}},j_{i}}}q_{i'}<p$ for $i<i'$. This is called the norm
decomposition and $r$ is called the remainder of $f$ with respect
to $N(x_{n_{j_1},j_1}), \dots , N(x_{n_{j_t},j_t})$. Notice that
$r$ is unique. If $f\in K[V]^{G}$ is an invariant, then the quotients $q_{i}$ for $1\le i\le
t$ and
the remainder $r$ are also invariant, see \cite[Proposition 2.1]{MR1897423}.

\begin{theorem} \label{transfer}
The algebra $K[V]^{G}/I^{G}$ is Cohen-Macaulay, and an hsop is
given by the set $\{N(x_{n_{j},j})+I^{G}\mid 1\le j\le l\}$. In
particular, we have $\depth(I^{G})=l+1$.
\end{theorem}

\begin{proof}
Let $\overline{f}:=f+I^{G}$ for $f\in K[V]^{G}$. We show that
$\overline{N(x_{n_{1},1})}, \ldots ,\overline{N(x_{n_l,l})}$ forms
a regular sequence for $K[V]^G/I^G$.  As its length $l$ equals the
dimension of $K[V]^{G}/I^{G}$, it follows that this ring is
Cohen-Macaulay. First we show that $\overline{N(x_{n_{i},i})}$ is
a $K[V]^G/I^G$-regular element for $1\le i\le l$. Assume
$fN(x_{n_{i},i})\in I^G$ for some invariant $f$. Then
$fN(x_{n_{i},i})=\tra (g)$ for some $g\in K[V]$. Consider the norm
decomposition $g=qN(x_{n_{i},i})+r$ of $g$ with respect to
$N(x_{n_{i},i})$. So $fN(x_{n_{i},i})=\tra
(qN(x_{n_{i},i})+r)=\tra (q)N(x_{n_{i},i})+\tra (r)$, hence
$0=(f-\tra(q))N(x_{n_{i},i})+\tra (r)$. Note that the group action
preserves the $x_{n_{i},i}$-degree, so we have
$\deg_{x_{n_{i},i}}(\tra(r))\le
\deg_{x_{n_{i},i}}(r)<p=\deg_{x_{n_{i},i}}(N(x_{n_{i},i}))$.
 So we get that $f-\tra (q)=0$  and $\tra (r)=0$. Therefore $f\in I^{G}$, and
$\overline{N(x_{n_{i},i})}$ is a $K[V]^G/I^G$-regular element. Assume now by
induction that $\overline{N(x_{n_{1},1})}, \ldots
,\overline{N(x_{n_{j-1},j-1})}$ is a
$K[V]^G/I^G$-regular sequence, and we have
\begin{equation}\label{eq1}
fN(x_{n_{j},j})=f_1N(x_{n_{1},1})+\cdots
+f_{j-1}N(x_{n_{j-1},j-1})+\tra (t),
\end{equation}
where $ f,  f_i\in K[V]^G$ for $1\le i\le j-1$ and $t\in K[V]$.
Consider the norm decompositions of $f$ and $t$ with respect to
$N(x_{n_{1},1}), \dots ,N(x_{n_{j-1},j-1})$. Since the quotient
and the remainder in the decomposition of $f$ are invariants, we
can replace $f$ by its remainder. As for $\tra (t)$, notice that
$\tra (t)$ and the transfer of the remainder of $t$ differ by a
combination of $N(x_{n_{1},1}), \dots ,N(x_{n_{j-1},j-1})$ by
invariants. So we can replace $\tra (t)$ with the transfer of the
remainder of $t$. Moreover, by considering the norm decomposition
of $f_i$ with respect to $N(x_{n_{1},1}), \dots
,N(x_{n_{i-1},i-1})$ for $1\le i\le j-1$, we replace $f_i$ with
its corresponding remainder. Therefore we may assume that
$\deg_{x_{n_{i'},i'}}f_i<p$ for $1\le i'<i$ and $1\le i\le j-1$.
Notice also that the degree of $f$ and $\tra (t)$ with respect to
any variable $x_{n_{i'},i'}$ is $<p$ for $1\le i'\le j-1$. Now
considering  Equation \eqref{eq1} as a polynomial equation in the
variable $x_{n_1,1}$ gives that $f_1=0$. Then comparing the
coefficients of $x_{n_2,2}$ gives $f_2=0$. Along the same way we
get $f_1=f_2=\cdots =f_{j-1}=0$. So Equation \eqref{eq1} becomes
$fN(x_{n_{j},j})=\tra (t)$. But since  $\overline{N(x_{n_{j},j})}$
is a $K[V]^G/I^G$-regular element, we have $f\in I^G$ as desired.
This shows that $\overline{N(x_{n_{1},1})}, \dots
,\overline{N(x_{n_l,l})}$ is a regular sequence. From
$\depth(K[V]^G/I^G)=l<\depth(K[V]^{G})=\min\{l+2,\dim_{K}(V)\}$
(we assume a non-trivial action) and Lemma~\ref{StandardCor1}, it
now follows that $\depth(I^{G})=l+1$.
\end{proof}

We also prove a reduction result for the depth of a module over
the invariant ring, which is based on the following lemma. The
statement is probably folklore, but for the convenience of the
reader and the lack of a reference, we provide a proof.

\begin{lemma}
Assume that $R$ is a graded affine ring and $M$ is a finitely
generated graded nonzero $R$-module. If $h_{1},\ldots,h_{r}\in
R_{+}$ form a homogeneous $M$-regular sequence and $I$ is a
homogeneous ideal of $R$ such that
$\sqrt{I+(h_{1},\ldots,h_{r})R}=R_{+}$, then
\[
\depth(M)=\grade(I,M/(h_{1},\ldots,h_{r})M)+r.
\]
\end{lemma}

\begin{proof}
As  the  homogeneous elements $h_{1},\ldots,h_{r}\in R_{+}$ form an $M$-regular sequence, we have that $\depth
(M)=\depth(M/(h_{1},\ldots,h_{r})M)+r$.
We show that \[
\grade (I, M/(h_{1},\ldots,h_{r})M)\ge\grade(R_{+}, M/(h_{1},\ldots,h_{r})M),\]
 as the reverse inequality is obvious.
Let $f_1, \dots , f_d\in R_{+}$ be a maximal homogeneous
$M/(h_{1},\ldots,h_{r})M$-regular sequence.
Since taking powers does not hurt regularity, we assume
that all elements in the sequence are contained in
$I+(h_{1},\ldots,h_{r})R$. Therefore for $1\le i\le d$  we can write
$f_{i}=g_{i}+b_{i}$ with homogeneous elements $g_i\in I$ and $b_{i}\in(h_{1},\ldots,h_{r})R$.
 Since  $b_{i}$ is
in the annihilator of \[ M/(h_{1},\ldots,h_{r}, f_1, \dots ,
f_{i-1})M=M/(h_{1},\ldots,h_{r}, g_1, \dots , g_{i-1})M
\]
it follows that $g_i$ is regular on $M/(h_{1},\ldots,h_{r}, g_1, \dots ,
g_{i-1})M$ as well. Hence the elements $g_1, \ldots , g_d$ of $I$ form
 an $M/(h_{1},\ldots,h_{r})M$-regular sequence.
\end{proof}

We recall that for any ideal $I$ of the invariant ring $K[V]^{G}$, we have
$\sqrt{I}=\sqrt{IK[V]}\cap K[V]^{G}$. This holds generally when $G$ is
a reductive group \cite[Lemma 3.4.2]{Newstead}, and an elementary proof for
finite groups can be found in \cite[Lemma 12.1.1]{EddyDavidBook}.

\begin{proposition}
 Let $M$ be a finitely generated  graded
$K[V]^{G}$-module on which the norms $N(x_{n_{1},1}),\ldots,N(x_{n_{l},l})$ form an
$M$-regular sequence.
Then \[
\depth (M)=\grade (I^G, M/(N(x_{n_1,1}),\ldots,N(x_{n_{l},l}))M)+l.\]
\end{proposition}

\begin{proof}
We have already mentioned that the zero set of $I^{G}$ is given by
$V^{G}=\bigoplus_{i=1}^{l}Ke_{n_{i},i}$. As for an element
$v=\sum_{i=1}^{l}\lambda_{i}e_{n_{i},i}\in V^{G}$ with
$\lambda_{i}\in K$, we have  $N(x_{n_{i},i})(v)=\lambda_{i}^{p}$,
the common zero set of
$I^{G}+(N(x_{n_{1},1}),\ldots,N(x_{n_{l},l}))$ is zero, hence
by the Nullstellensatz
$\sqrt{(I^{G}+(N(x_{n_{1},1}),\ldots,N(x_{n_{l},l})))K[V]}=K[V]_{+}$.
From the paragraph before the proposition we obtain that the
radical ideal of $I^{G}+(N(x_{n_{1},1}),\ldots,N(x_{n_{l},l}))$
equals $K[V]^{G}_{+}$, and the lemma above applies.
\end{proof}

Examples where the proposition applies include the case $l=1$ and $M=I$ a
nonzero homogenous ideal of $K[V]^{G}$. The corollary also applies for arbitrary
$l$ and $M=K[V]^{G}$ by Proposition \ref{regSeqCp}. In the ``non-trivial''
cases where $\depth(K[V]^{G})=l+2$, it follows from
$\depth(K[V]^{G})=\grade(I^{G},K[V]^{G}/(N(x_{n_{1},1}),\ldots,N(x_{n_{l},l})))+l$,
that
\[
\grade(I^{G},K[V]^{G}/(N(x_{n_{1},1}),\ldots,N(x_{n_{l},l})))=2.
\]
Therefore, there is a maximal $K[V]^{G}$-regular sequence
consisting of the $l$ norms and two transfers. Also compare with
the known fact that $\grade(I^{G},K[V]^{G})=2$ in these cases, see
\cite[Propositions 20 and 22]{CampEtAl}. As
$\depth(K[V]^{G})=l+2$, this also shows that  $\depth(M)\ne
\grade(I^{G}, M)$ in general.

\begin{ack}
We thank T\" {u}bitak for
funding a visit of the first author to Bilkent University, and
Gregor Kemper for inviting the second author to TU
M\"unchen. Special thanks go to Fabian Reimers for many inspiring
conversations, and in particular for showing Lemma
\ref{StandardCor1} to us.
\end{ack}

\bibliographystyle{plain}
\bibliography{OurBib}
\end{document}